\documentclass{article} 
\usepackage{amsmath,amssymb,mathtools,amsthm}
\usepackage[all]{xy}
\usepackage{enumerate}
\usepackage[utf8]{inputenc}

\makeatletter
\newtheorem*{rep@theorem}{\rep@title}
\newcommand{\newreptheorem}[2]{%
\newenvironment{rep#1}[1]{%
 \def\rep@title{#2 \ref{##1}}%
 \begin{rep@theorem}}%
 {\end{rep@theorem}}}
\makeatother

 \newtheorem{thm}{Theorem}[section]
      \newtheorem{lemma}[thm]{Lemma}
      \newtheorem{prop}[thm]{Proposition}
      \newtheorem{cor}[thm]{Corollary}
      \newtheorem*{thm*}{Theorem}
     \theoremstyle{definition}
      
      \newtheorem{eg}{Example}[section]
      \newtheorem*{eg*}{Example}
      \theoremstyle{remark}
      \newtheorem{rk}{Remark}[section]

\newcommand{\Zee}{\mathbb{Z}}
\newcommand{\Qee}{\mathbb{Q}}
\newcommand{\Cee}{\mathbb{C}}

\newcommand{\pa}{p_{\alpha}}
\newcommand{\da}{d_{\alpha}}
\newcommand{\vo}{v^{\text{old}}}
\newcommand{\vn}{v^{\text{new}}}
\newcommand{\exactmid}{\mid\mspace{-2mu}\mid}

\title{Properties of solutions to Pell's equation over the polynomial ring\thanks{This research was supported by ERC grant n$^{\text{o}}$ 670239}}
\author{Nikoleta Kalaydzhieva\thanks{Email: zcahndk@ucl.ac.uk; OrciD: 0000-0002-9393-283X} \\ \emph{University College London}}

\begin{document} 
\maketitle{}

\begin{abstract}


 In the classical theory, a famous by-product of the continued fraction expansion of quadratic irrational numbers $\sqrt{D}$ is the solution to Pell's equation for $D$. It is well-known that, once an integer solution to Pell's equation exists, we can use it to generate all other solutions $(u_n,v_n)_{n\in\Zee}$. Our object of interest is the polynomial version of Pell's equation, where the integers are replaced by polynomials with complex coefficients. We then investigate the factors of $v_n(t)$. In particular, we show that over the complex polynomials, there are only finitely many values of $n$ for which $v_n(t)$ has a repeated root. Restricting our analysis to $\Qee[t]$, we give an upper bound on the number of ``new'' factors of  $v_n(t)$ of degree at most $N$. Furthermore, we show that all ``new'' linear rational factors of $v_n(t)$ can be found when $n\leq 3$, and all ``new'' quadratic rational factors when $n\leq 6$. 
\end{abstract}

\section*{Introduction}
Pell's equation is defined to be
\begin{align}\label{nPell}
	x^2-Dy^2=1,
\end{align}
 and classically solved in positive integers $x=u,\ y=v$, for a given non-zero positive integer $D$, which is not a square. 

If we take the solution $(u,v)$ in which $v$ is the smallest positive integer, then we can use it to generate all other solutions to (\ref{nPell}) by 
\begin{align}\label{solgen}
	u_n\pm v_n\sqrt{D}=\pm\left(u+v\sqrt{D}\right)^n.
\end{align}

In this thesis we are interested in the polynomial analogue to the integers case. Indeed, we study solutions $u(t),v(t)\in\Cee[t]$, with $v\neq 0$, to Pell's equation for a polynomial $D(t)$ with coefficients in $\Cee$. If (\ref{nPell}) is solvable, we take its \emph{fundamental solution}, the one in which $v$ has minimal degree, and obtain all other solutions $(u_n(t),v_n(t))_{n\in\Zee}$ in the same way as in the classical case, using (\ref{solgen}).

Our goal, in the first part of the thesis, is to better understand the polynomials $v_n(t)$ that arise in the solutions of Pell's equation when $D(t)\in\Qee[t]$. In the classical case, when $D$ is a square-free, positive integer, it has been of great interest to factor $v_n\in\Zee$, see \cite{Brillhart1988}. Additionally, Lehmer \cite{Lehmer1928} showed that in certain cases $v_n\in\Zee$ factors into many parts. However, we will see that in the polynomial case the factors over $\Qee[t]$ of $v_n(t)$ are very controlled.

Similar to the integers case, we have that $gcd(v_n(t),v_m(t))=v_{gcd(m,n)}(t)$. In particular, if $m\mid n$ then $v_m(t)\mid v_n(t)$, and $v_1(t)\mid v_n(t)$, for all $n$. Furthermore, we will also show that $gcd\left(v_m(t),\ v_n(t)/v_m(t)\right)=1$, which is not always the case over the integers: there, if a prime $p\mid v_n$, but $p^2\nmid v_n$, then $p^2\mid v_{np}$; in other words $p\mid gcd\left(v_{n},\ v_{np}/v_n\right)$.

Over $\Cee[t]$, we can write $v_n(t)=\vo_n(t)\vn_n(t)$, where $\vo_n(t)$ is a product of the factors of $v_n(t)$ that also divide $v_m(t)$ for some $m<n$, and $\vn_n(t)$ are the remaining factors, including multiplicity. Then we obtain
\begin{align*}
	v_n(t)&=\prod_{m\mid n}\vn_m(t) \text{ and}\\
	\vn_n(t)&=\prod_{m\mid n}v_m^{\mu\left(\frac{n}{m}\right)}(t),
\end{align*}
where the latter follows from the product form of Möbius inversion.
Furthermore, the $\vn_n$ are pairwise co-prime, so we study their factors. Our first goal is to understand whether $\vn_n(t)$ ever has any repeated factors. It turns out that for any fixed $D(t)\in\Cee[t]$, there are only finitely many $n$ for which $\vn_n(t)$ has repeated factors. This comes out as a consequence of

\begin{thm}
 	For any polynomial $D(t)\in\Cee[t]$, for which the associated Pell's equation has a fundamental solution $(u(t),v(t))$, we define
 	\begin{align*}
 	R(D):=\{\alpha\in\Cee : (t-\alpha)^2 \mid \vn_n(t)\text{ for some }n\}.
 	\end{align*}
 	Then $\# R(D)\leq \deg{u}-1$.
 \end{thm}

The proof actually gives us a finite algorithm that yields all repeated roots of $\vn_n$ for all $n$. It turns out that they come from the factors of $u^{\prime}(t)$. Suppose \\$(t-\alpha)^k\exactmid u^{\prime}(t)$, for $k\geq 1$. We show that if $u(\alpha)=\cos\frac{\pi r}{n}$, for some $r<n$, with $(r,n)=1$, then $(t-\alpha)^{k+1}\exactmid \vn_n(t)$. 

If we then restrict ourselves to working over $\Qee[t]$, the repeated root $\alpha$ must come from a field extension of degree $\da$, satisfying $\varphi(2n)/2<\da<\deg{u}$. So using results on the growth order of the Euler totient function \cite{Sandor2006} we show that if $\vn_n(t)$ has a repeated root, we must have $n\ll d\log\log{d},\text{ where }d=\deg{u}$. 
We then proceed to look at the degrees of the irreducible factors of $\vn_n(t)$ when $u,v,D\in\Qee[t]$, and obtain various Galois theoretic results.

\begin{thm}
Let $N$ be a positive integer, and define $$I(N):=\{P(t)\in\Qee[t],\text{ irreducible} : \deg{P}\leq N,\ P(t)\mid \vn_m(t)\text{ for some }m\}.$$ Then $\# I(N)\leq 10N\deg{u}$, for $N$ large enough.
\end{thm}

If we want a result that holds for all $N$, we can get $\# I(N)\leq 4N^2\deg{u}$.
Specialising to factors of certain degree, we show:
\begin{thm}
\leavevmode
\begin{enumerate}
	\item There are no linear polynomials with coefficients in $\Qee$ that divide $\vn_n(t)$, for $n\geq 4$.
	\item There are no quadratic polynomials with coefficients in $\Qee$ that divide $\vn_n(t)$, for $n\geq 7$. 
\end{enumerate}
\end{thm}

Both of these results are best possible. To see this, we present the following example: 
\begin{eg}
	Let $D(t)=t^2-1$, then its Pell's equation has a fundamental solution $(t,1)$. The only linear factors of $v_n(t)$, are $\vn_2=2t$, $\vn_3=2t\pm 1$. For $n\geq 4$, there are no linear factors over the rational numbers. Furthermore, the only quadratic irreducible factors are $\vn_4=2t^2-1$, $\vn_5=4t^2\pm 2t-1$ and $\vn_6=4t^2-3$. For $n\geq 7$, $\vn_n$ has no irreducible quadratic factors in $\Qee[t]$. We will give examples for all possibilities in section \ref{factgivendegree}.
\end{eg}
Moreover, we show that for a repeated root of $\vn_n(t)$ to be of a prime degree $p$, then $p$ lies in a subset of the primes that has density $0$.  
\begin{thm}
Suppose $D(t)\in\Qee[t]$, which has Pell's equation with fundamental solution $(u,v)$. If the polynomials $\vn_n(t)$ for $n>3$ have a repeated root $\alpha$ of an odd prime degree $\da$, then either $n=2\da+1$ is also prime, or $n=9$ in which case $\alpha$ is cubic.
\end{thm}

Restricting our investigation to polynomials with integer coefficients:
\begin{cor}
	For $u, v, D\in\Zee[t]$, the polynomials $\vn_n(t)$ for $n>3$ have no repeated factors that are quadratic polynomials with integer coefficients.
\end{cor}


Solving the polynomial Pell's equation is not completely analogous to the classical case: for instance, there are certain $D(t)\in\Cee[t]$ with corresponding Pell's equation that has no non-trivial solutions. This is obvious when $D(t)$ has odd degree, since the term of highest degree cannot be cancelled. Therefore, we fix $D(t)$ to be a polynomial of degree $2d$ and look for solutions to \eqref{nPell} amongst polynomials with complex coefficients. However, there are examples when $D(t)$ has even degree, for instance $D(t)=t^4+t+1$, where it is less obvious why the corresponding Pell's equation is not solvable. We will call polynomials $D(t)$ for which \eqref{nPell} has a non-trivial solution \emph{Pellian}. 
We can use our understanding of the factors of $v_n(t)$, for a given Pellian polynomial $D(t)$, together with the following lemma, to construct new Pellian polynomials.
\begin{lemma}
	 The polynomial $F^2D(t)$ is Pellian if and only if $D(t)$ is a Pellian polynomial with solutions $(u_n(t),v_n(t))_{n\in\Zee}$, and $F(t)\mid v_n(t)$, for some $n$.
\end{lemma}
Furthermore, recall that if we restrict $D(t), \ u(t)\text{ and } v(t)\in\Qee[t]$, then there are only finitely many factors $F(t)\in\Qee[t]$ of $v_n(t)$, of a given degree. Therefore, there are infinitely many families of infinitely many polynomials of the form $F^2D(t)$ that are not Pellian. This contrasts with the situation over the integers, where for any non-square positive integer of the form $F^2D$, Pell's equation is solvable.


Organisation of this paper: In section \ref{multisol} we develop the factorisation of $v_n(t)$ as a product of old and new factors. We also, show that for $n>1$, we can write $\vn_n(t)$ as an integral polynomial in $u$. In section \ref{2} we study the bounds on the number of factors and repeated factors of $\vn_n(t)$. Firstly, for the factors of $\vn_n(t)$, when $D\in\Qee[t]$, we show there are different bounds on the maximum, depending on the parity of $n$. We next prove that for $D\in\Cee[t]$ there are no repeated roots of $\vn_n(t)$ for $n$ large enough. Furthermore, if we restrict to the rational polynomials, we compute an asymptotic on $n$, beyond which $\vn_n(t)$ could not  have repeated factors. In section \ref{3} we turn our attention to the degrees of the factors and repeated factors of $\vn_n(t)$. We prove a bound on the number of irreducible polynomials up to a degree $N$, that divide any $\vn_n(t)$. Furthermore, using Galois theoretic methods, we discard $\alpha\in\Cee$ of  degree equal to certain primes, as being possible repeated root. Finally we study in more detail the occurrence of quadratic irrationals as repeated roots. In the appendix, we show an application of the finiteness results on the factors of $\vn_n(t)$ to the construction of non square-free Pellian polynomials.

\section{Factorisation properties of $v_n(t)$}\label{multisol}

Let $D\in\Cee[t]$ be a Pellian polynomial with \emph{fundamental solution} $(u,v)$, with $v\neq 0$ of smallest degree. Then, for each integer $n$ greater than $1$, we obtain a new pair of polynomials $(u_n,v_n)$, satisfying the same Pell's equation, by 
\[
u_n +v_n\sqrt{D} = ( u+v\sqrt{D} )^n.
\]
Furthermore, we can extend this definition to the negative integers by applying the identity $(u+v\sqrt{D})^{-1}=u-v\sqrt{D}$. This equivalence also gives 
\[
u_n-v_n\sqrt{D} = ( u-v\sqrt{D} )^n,
\]
so that
\[
v_n  = \frac{ ( u+v\sqrt{D} )^n-( u-v\sqrt{D} )^n} {2\sqrt{D}} .
\]
By definition, $v_1=v$ and, as $a^n-b^n=\prod_{\xi:\ \xi^n=1} (a-b\xi)$, we have
\[
v_n = v_1 \prod_{\substack{\xi:\ \xi^n=1,\\ \xi\ne 1}} \left(( u+v\sqrt{D} )-\xi ( u-v\sqrt{D} )\right) .
\]
Note that $\xi$ is a root of $t^n-1$, which can be written as a product of irreducible factors as $\prod_{d|n} \phi_d(t)$, where $\phi_d(t)$ denotes the cyclotomic polynomials. Therefore
\begin{align}\label{vfactor}
v_n=v_1\prod_{m|n, m>1} \psi_m,
\end{align}
where
\begin{align}\label{psim}
\psi_m := \prod_{\xi:\  \phi_m(\xi)=0} \left(( u+v\sqrt{D} )-\xi ( u-v\sqrt{D} )\right).
\end{align}



We can exploit that the product in (\ref{psim}) is taken over the roots of the cyclotomic polynomials, to show that $\psi_m$ does not depend on $v$ or $\sqrt{D}$, but rather:
\begin{lemma}\label{integralpoly}
	For all integers $m$ greater than $1$, we have $\psi_m\in\Zee[2u]$.
\end{lemma}

\begin{proof}
	
Firstly, note that $\phi_2(t)=t+1$ and so $\psi_2=2u$. Now assume that $m>2$, then $\phi_m$ is always of even degree and thus the $\xi$'s come in conjugate pairs. We will study $\psi_m$ by pairing up the terms for $\xi$ and $\overline{\xi}$.
Together we have
\[
\left(( u+v\sqrt{D} )-\xi ( u-v\sqrt{D} )\right) \left(( u+v\sqrt{D} )-\overline{\xi}( u-v\sqrt{D} )\right) =
(2u)^2-(\xi+2+\overline{\xi}).  
\]
This implies that the product in $\psi_m$ is a product over conjugates and therefore belongs to $\Zee[2u]$. 
\end{proof}

To better understand the factors of $v_n$ it suffices to factorise $\psi_m$ over $\Cee[u]$
\begin{lemma}\label{vn}
	The polynomials $\psi_m(u)$ have roots $\cos\left(\frac{r\pi }{m}\right)$, for $1\leq r< m$ and $(r,m)=1$. Namely, 
	\begin{align*}
	\psi_m(u)=2^{\varphi(m)}\prod_{\substack{1\leq r< m\\ (r,m)=1}}\left(u-\cos\frac{r\pi}{m}\right).
\end{align*}
\end{lemma}

\begin{proof}
The roots of $\phi_m$ are $e( \frac rm)$ for $1\leq r<m$ with $(r,m)=1$, and for each conjugate pair we take $\xi = e( \frac rm)$ with $1\leq r<m/2$. Here $e(a):=e^{2a\pi i}$ and using Euler's formula $e(a)=\cos{2a}+i\sin{2a}$, we have 
\begin{align*}
\xi+2+\overline{\xi}=e\left( \frac rm\right)+2+e\left( -\frac rm\right)=2\cos{\frac{2r\pi}{m}}+2=4\cos^2{\frac{r\pi}{m}}.
\end{align*}
The latter equality comes from the double angle formula. Consequently, 
\[
\psi_m=\prod_{\substack{1\leq r< m/2\\ (r,m)=1}} \left((2u)^2-\left(2\cos{\frac{r\pi}{m}}\right)^2\right) .
\]
Now $\cos{\frac{r\pi}m}=-\cos{\frac {(m-r)\pi}{m}}$, thus 

\[
\psi_m=\prod_{\substack{1\leq r< m/2\\ (r,m)=1}} \left(2u-2\cos{\frac{r\pi}{m}}\right)\left(2u-2\cos{\frac{(m-r)\pi}{m}}\right).
\]

Observe that $m/2<s< m$ with $(s,m)=1$ if and only if $s=m-r$, where $0<r\leq m/2$ with $(r,m)=1$, and so the above becomes
\[
\psi_m=\prod_{\substack{1\leq r< m\\ (r,m)=1}} \left(2u-2\cos{\frac{r\pi}{m}}\right)
=2^{\varphi(m)} \prod_{\substack{1\leq r< m\\ (r,m)=1}} \left(u-\cos\left(\frac{ r\pi }{m}\right)\right).
\]
\end{proof}
Furthermore, since the cosines are distinct, as $m$ ranges in the natural numbers and $r$ in the given interval, we have 
\begin{lemma}\label{copr}
The polynomials $\psi_m$ for $m>1$, have no common roots.
\end{lemma}

If we instead we restrict ourselves to working over polynomials in $u$ with coefficients in $\Qee$, we obtain:
\begin{thm}\label{modd}
	The polynomials $\psi_m$ are irreducible over $\Qee[u]$ for $m$ even and split into two irreducible factors of degree $\varphi(m)/2$, if $m$ is odd. Namely, for odd integers $m>1$
	\begin{align*}
	 	\psi_m(u)=(-1)^{\varphi(m)/2}\psi^*_m(u)\psi^*_m(-u),
	 \end{align*} 
	 where 
	 \begin{align*}
	 	\psi_m^*(u)=2^{\varphi(m)}\prod_{\substack{1\leq q< m/2\\ (q,m)=1}}\left(u-\cos\frac{2q\pi}{m}\right).
	 \end{align*}
\end{thm}
\begin{proof}
To see this, recall from Lemma \ref{vn} that $\psi_m(u)$ have roots $\alpha_r=e( \frac r{2m})+e( \frac {-r}{2m})$, where $1<r<m$ and $(r,m)=1$. The field $\Qee(\alpha_r)$ is the real subfield of $\mathbb Q\left(e( \frac r{2m})\right)$ of relative degree 2. Then we have the following tower of extensions.
\[
\xymatrix{
	\Qee\left(e( \frac r{2m})\right)\ar@{-}[d]^{2} \ar@{-}@/_2pc/[dd]_{\varphi(2m)}\\
	\Qee(\alpha_r)\ar@{-}[d]\\
	\Qee
}
\] 

Therefore $\Qee(\alpha_r)$ has degree $\varphi(2m)/2$ over the rationals, and if $m$ is even, this equals $\varphi(m)$, implying $\psi_m$ is irreducible. However, if $m$ is odd, then $\Qee(\alpha_r)$ has degree $\varphi(m)/2$ over the rationals, and so $\psi_m$ must factor into two irreducible polynomials of degree $\varphi(m)/2$. In particular, if $r=2q$, then $e( \frac r{2m})=e( \frac q{m})$, and if $r$ is odd then write $r=m-2q$ and so 
$e( \frac r{2m})=e( \frac {m-2q}{2m})=-e( \frac {-q}{m})$. We deduce that
\[
\psi_m=\prod_{\substack{1\leq q<m/2\\ (q,m)=1}} \left(2u-\left(e\left( \frac q{m}\right)+e\left( \frac {-q}{m}\right)\right)\right)\left(2u+\left(e\left( \frac q{m}\right)+e\left( \frac {-q}{m}\right)\right)\right).
\]
That is, $\psi_m(u)=(-1)^{\varphi(m)/2} \psi_m^*(u)\psi_m^*(-u)$, where
\begin{align*}
\psi_m^*(u) :&= \prod_{\substack{1\leq q\leq m/2\\ (q,m)=1}}
 \left(2u-\left(e\left( \frac q{m}\right)+e\left( \frac {-q}{m}\right)\right)\right)\\
&=2^{\varphi(m)} \prod_{\substack{1\leq q\leq m/2\\ (q,m)=1}} \left(u-\cos\left( \frac{2q\pi }{m}\right)\right)
\end{align*}
is irreducible.
\end{proof}

In summary, we have written $v_n=v\prod_{\substack{m\mid n\\ m>1}}\psi_m(u)$, with $\psi_m$ integral polynomials only depending on $u$. Observe further that, if $\alpha=\cos\left(r\pi /n\right)$ with $(r,n)=1$, then $\psi_n(\alpha)=0$, but $\psi_m(\alpha)\neq 0$ for all integers $m$ smaller than $n$. That is $\psi_n$ picks out the \emph{``new roots''} of $v_n$. Therefore over $\Cee[t]$, we can write $v_n(t)=\prod \vn_n(t)\vo_n(t)$, where  $\vn_n(t)=\psi_n(u(t))$.

\section{Bounding the number of factors of $\vn_n(t)$}\label{2}
We will use the observation that $\vn_n(t)=\psi_n(u(t))$, together with the factorisation results in section \ref{multisol}, to study the factors and repeated factors of $\vn_n(t)$, as a polynomial in $t$. Whenever necessary, we will adopt the notation $\deg_x(f)$ to indicate the degree of $f$
 as a polynomial in $x$. 

\subsection{Bounds on the number of factors}

\begin{lemma}\label{cor}
Let $\pi(u)\in\Qee[u]$ be a product of $k$ irreducible factors in $\Qee[u]$. Then $P(t):=\pi(u(t))\in\Qee[t]$ has no more than  $k\deg{u}$ irreducible factors over $\Qee[t]$.
\end{lemma}

\begin{proof}
We prove this when $\pi(u)$ is irreducible over $\Qee[u]$, and whenever it is reducible, we multiply the result by the number of irreducible factors of $\pi$.

Suppose $\pi(u)$ is irreducible over $\Qee[u]$. Let $A$ be a root of $\pi(u)$ over $\Cee$. Then $\pi(A)=0$, so there exists $\alpha\in\Cee$, such that $P(\alpha)=\pi(A)=0$. In particular $\alpha$ is a root of $u(t)-A$ and we have the tower of extensions
	\[
\xymatrix{
  \Qee(\alpha)\ar@{-}[d]_{1\leq} \ar@{-}@/^2pc/[dd]^{\geq\deg_u\pi}\\
  \Qee(A)\ar@{-}[d]_{\deg_{u}\pi}\\
  \Qee
}
\]
 yielding that the minimal polynomial of $\alpha$ over the rationals must be of degree at least $\deg_u\pi$. Now $\deg_t{P}=\deg_u\pi\deg{u}$, therefore $P(t)$ has at most 
 $\deg{u}$ irreducible factors over $\Qee[t]$.
\end{proof}

\begin{prop}\label{general+}
	Given $D(t)\in\Qee[t]$, let $(u_n(t),v_n(t))$ be the $n^{\text{th}}$ solution to Pell's equation $X^2(t)-D(t)Y^2(t)=1$, generated by the fundamental solution $(u(t),v(t))$. Then the new factors of $v_n(t)$ are at most $\deg u$ if $n$ is even and $2\deg u$ if $n$ is odd.	
\end{prop}

\begin{proof}
	 In the introduction to this chapter, we observed that $\vn_n(t)=\psi_n (u(t))$. Furthermore, from Theorem \ref{modd}, the polynomials $\psi_n(u)$ split into two irreducible factors over $\Qee[u]$ if $n$ is odd and are irreducible if $n$ is even. The result is then a consequence of Lemma \ref{cor}, for $\psi_n(u)$ and $u=u(t)\in\Qee[t]$.
\end{proof}

\subsection{Bounds on the number of repeated factors}\label{boundreprootssec}

Suppose $\alpha\in\Cee$ is a repeated root of some $v_n(t)$. That is $(t-\alpha)^2\mid v_n(t)$, and since the $\vn_n$ have no common roots, we must have $(t-\alpha)^2\mid\vn_m(t)$, for some $m\mid n$. Hence, to understand the repeated factors of $v_n(t)$, it suffices to consider the repeated factors of $\vn_n(t)$. In this section we study their existence and dependence on $n$. 

\begin{thm}\label{boundrep}
 	For any Pellian polynomial $D(t)\in\Cee[t]$, with fundamental solution $(u(t),v(t))$, we define
 	\begin{align*}
 	R(D):=\{\alpha\in\Cee : (t-\alpha)^2 \mid v_n(t)\text{ for some }n\}.
 	\end{align*}
 	Then $\# R(D)\leq \deg{u}-1$.
 \end{thm}

 \begin{proof}
 		By the discussion in the introduction of subsection \ref{boundreprootssec} it suffices to consider repeated factors of $\vn_n(t)$. We first study $\vn_1(t)=v(t)$, as it cannot be expressed as a polynomial in $u$. Suppose $(t-\alpha)^2\mid v(t)$, then $(t-\alpha)^4\mid D(t)v^2(t)$. Since $(u(t),v(t))$ is a solution to Pell's equation, we must have 
 		\begin{align*}
 			(t-\alpha)^4&\mid u^2(t)-1\\
 			\Rightarrow(t-\alpha)^3&\mid u(t)u^{\prime}(t).
 		\end{align*}
Observe that, this implies that $(t-\alpha)^3$ is a factor of $u^{\prime}(t)$, since $u(\alpha)=\pm 1$.
Next suppose that $(t-\alpha)$ is a repeated root of $\vn_n(t)$, for $n>1$. Therefore we must have $(t-\alpha)^2\mid u(t)-\cos{\pi r/n}$, for some positive integer $r<n$, co-prime to $n$. Then $(t-\alpha)$ must be a factor of $u^{\prime}(t)$. In summary the repeated factors of $\vn_n(t)$ over $\Cee[t]$, arise from roots of $u^{\prime}(t)$, and there are at most $\deg{u}-1$ of them.
 \end{proof}

 \begin{cor}
 	For any Pellian $D(t)\in\Cee[t]$, there are only finitely many $n$, for which $\vn_n(t)$ has repeated factors.
 \end{cor}

 \begin{proof}
 	In the proof of theorem \ref{boundrep} we showed that if $\alpha$ is a repeated root of $\vn_n(t)$ for any $n$, then $(t-\alpha)\mid u^{\prime}(t)$. Since $u(t)$ is a polynomial, it has finitely many roots over $\Cee$, and therefore there are only finitely many $n$ for which $\vn_n(t)$ has repeated factors.
 \end{proof}
 The proof of theorem \ref{boundrep} gives us a method of explicitly finding all repeated roots of $\vn_n(t)$. Namely, suppose $(t-\alpha)^k\exactmid u^{\prime}(t)$, and if further $u(\alpha)=\cos(\pi r/n)$, for some $r<n$, co-prime to $n$, then $(t-\alpha)^{k+1}\exactmid \vn_n(t)$. Then the repeated factors of $\vn_n(t)$ must arise from repeated roots $\alpha\in\Cee$ of $u(t)-\cos(\pi r/n)$.

We now focus our attention to repeated factors of $\vn_n(t)$ over the rational numbers. Restricting the factors to $\Qee[t]$ implies that the repeated root $\alpha$ must come from a field extension of degree $\da$, satisfying:
\[
\xymatrix{
	\Qee(\alpha)\ar@{-}[d]^{} \ar@{-}@/_3pc/[dd]_{\da}\\
	\Qee\left(\cos\left(\frac{\pi r}{n}\right)\right)\ar@{-}[d]^{\varphi(2n)/2}\\
	\Qee
}
\] 
 \begin{thm}
 	For any Pellian $D(t)\in\Qee[t]$, with fundamental solution $(u(t),v(t))$, if $\vn_n(t)$ has a repeated root, then $n\ll d\log\log{d}$, with $d=\deg{u}$.

 \end{thm}
\begin{proof}

Let $\alpha\in\Cee$ algebraic of degree $\da$ be a repeated root of multiplicity $k>1$ of $\vn_n(t)$. Then from the discussion before the statement, together with the Tower law, we deduce that if $\alpha$ is a repeated root of multiplicity $k>1$, of $\vn_n(t)$, then $d=\deg{u}>\da>\varphi(2n)/2$. The Euler totient function satisfies the asymptotic formula $\varphi(n)\gg n/\log\log{n}$, see \cite{Sandor2006}. Combining the two and simplifying appropriately, yields the desired asymptotic $n\ll d\log\log{d}$.
\end{proof}

\section{The degrees of the factors of $\vn_n(t)$}\label{3}
In this section we let $D(t)\in\Qee[t]$ be a Pellian polynomial with fundamental solution $(u,v)$ and we wish to study the degrees of the rational irreducible factors of $\vn_n(t)$.
\subsection{Factors of given degree}\label{factgivendegree}
We once again exploit the fact that $v_n(t)$ can also be written as the composition of two polynomials. This time with the help of the following technical lemma
\begin{lemma}\label{mindegree}
	Let $P,Q\in\Qee[X]$. Any rational factor of $P(Q(X))$ is of degree at least the degree of the smallest rational factor of $P(X)$.
\end{lemma}

\begin{proof}
	Let $\alpha\in\Cee$ be the root of $P(X)$ of smallest degree and let $\beta\in\Cee$ be an arbitrary root of $P(Q(X))$. Therefore $Q(\beta)$ will be a root of $P(X)$, and by the minimality of $\alpha$, we will have $\deg{Q(\beta)}\geq \deg{\alpha}$. A final observation that $\deg{\beta}\geq \deg{Q(\beta)}$ completes the proof.
\end{proof}
Using lemma \ref{factgivendegree} for $\vn_n(t)=\psi_n(u(t))$, together with an asymptotic bound on the number of solutions to the equation $\varphi(n)=m$, we get
\begin{thm}\label{bound}
	Let $N$ be a positive integer and define $$I(N):=\{P(t)\in\Qee[t],\text{ irreducible}: \deg{P(t)}\leq N,\ P(t)\mid \vn_n(t)\text{ for some }n\}.$$ For $N$ sufficiently large, $\# I(N)\leq 10N\deg{u}$.
\end{thm}

\begin{proof}
Suppose that $P(t)$ is a factor of $\vn_n(t)$. By lemma \ref{mindegree}, $\varphi(2n)/2\leq\deg{P}$, since any rational factor of $\vn_n(t)=\psi_n(u(t))$, must be at least the degree of the smallest rational factor of $\psi_n(t)$. Fix the degree of $P$ to be at most some positive integer $N$. To estimate the number of elements of $I(N)$ it suffices to compute the number of integers $n$ that satisfy the inequality $\varphi(2n)/2\leq N$ and multiply it by $\deg{u}$ or $2\deg{u}$ depending on the parity of $n$. Observe that
\begin{align*}
	\#\{n:\varphi(2n)\leq 2N\}=\#\{n \text{ even}:\varphi(n)\leq N\} +\#\{n\ \text{odd}:\varphi(n)\leq 2N\}.
\end{align*}
From \cite{Sandor2006} we have that $\#\{m:\varphi(m)\leq x\}=\frac{\zeta(2)\zeta(3)}{\zeta(6)}x+R(x)$, with $R(x)$ of order at most $x/(\log x)^l$, for any positive $l$. Thus the number of elements in the set is at most $2x$ for $x$ sufficiently large.. Hence $\# I(N)\leq 2N\deg{u}+8N\deg{u}$ for $N$ sufficiently large.
\end{proof}

If we then want an inequality that holds for all positive integers of $N$, we simply use a bound on the Euler totient function. However that makes the bound in theorem \ref{bound} much worse for large values of $N$.

\begin{prop}
Let $N$ be a positive integer, then $\# I(N)\leq 4N^2\deg{u}$.
\end{prop}

\begin{proof}
	 Similarly to the proof of theorem \ref{bound}, if $P(t)$ is a factor of $\vn_n(t)$, then $\varphi(2n)/2\leq\deg{P}$, since the smallest factor of $\vn_n(t)$ is of degree $\varphi(2n)/2$. Hence any irreducible rational factor of $\vn_n(t)$, of degree up to $N$, satisfies $\varphi(2n)<2N$. Using the lower bound of the Euler totient function $\varphi(n)\geq \sqrt{n}$ and simplifying we obtain $n\leq 2N^2$. Now from lemma \ref{mindegree}, $\vn_n$ has at most $\deg{u}$ irreducible factors for each $n$ even, and $2\deg{u}$ irreducible factors for each $n$ odd. Hence $\# I(N)\leq 4N^2\deg{u}$.
\end{proof}

\begin{cor}\label{factors}
	Let $D(t)\in\Qee[t]$ be a square-free Pellian polynomial and $N$ a positive integer. We define 
	\begin{align*}
		J(N):=\{P(t)\in\Qee[t],\text{ irreducible}: \deg{P(t)}= N,\ P(t)\mid \vn_n(t)\text{ for some }n\}.
	\end{align*}
	Then $\#J(N)<\infty$.       
\end{cor}

Furthermore, we can obtain more explicit results, by considering factors of specific degree  
\begin{thm}\label{linear}
Suppose $D(t)\in\Qee[t]$ is a Pellian polynomial with fundamental solution $(u,v)$. 
\begin{enumerate}
		\item There are no linear polynomials with coefficients in $\Qee$ that divide $\vn_n(t)$, for $n\geq 4$.
	\item There are no quadratic polynomials with coefficients in $\Qee$ that divide $\vn_n(t)$, for $n\geq 7$. 
\end{enumerate}
\end{thm}

\begin{proof}
	Similarly to the above two theorems, we use the fact that the smallest factor of $\vn_n(t)$ is of degree $\varphi(2n)/2$. 
\begin{enumerate}
	\item Therefore if $\varphi(2n)/2>1$ there cannot be a rational factor of $\vn_n(t)$. Furthermore, we need $\varphi(n)>1$ and $n$ even or $\varphi(n)>2$ and $n$ odd. If $n\geq 4$ both of those inequalities are satisfied.
	\item For no quadratic rational factors of $\vn_n(t)$, we need $n$ to satisfy $\varphi(2n)/2>2$. In particular, we want $\varphi(n)>2$ and $n$ even and $\varphi(n)>4$ and $n$ odd. For $n\geq 7$ the inequalities hold.
	\end{enumerate}
\end{proof}

Both of these results are actually best possible. To see this, consider the following example: 
\begin{eg}
	Let $D(t)=t^2-1$, then $(t,1)$ is the smallest solution to the corresponding Pell's equation and thus it generates all the others. The only linear factors of $v_n(t)$, are $\vn_2=2t$, and the factors of  $\vn_3$, i.e $2t\pm 1$. Furthermore, the only quadratic irreducible factors are $\vn_4=2t^2-1$, $\vn_6=4t^2-3$ and the factors of $\vn_5$, namely $4t^2\pm 2t-1$.
	For $D(t)=t^4+t^2$, with fundamental solution $(2t^2+1,\ 2)$, we have $\vn_2=2(2t^2+1)$ and $\vn_3=(4t^2+1)(4t^2+3)$, having quadratic irreducible factors.
	Suppose $D(t)=t^8+4t^6+6t^4+5t^2+2$, this has fundamental solution $\left(2t^6+6t^4+6t^2+3,\ 2(t^2+1)\right)$, then $\vn_1=2(t^2+1)$ is a quadratic irrational. 
	\end{eg}

\subsection{Repeated factors of given degree}\label{secrepfacdeg}

Suppose that $\alpha\in\Cee$ is algebraic with minimal polynomial $\pa(t)$ of degree $\da$ over the rational numbers. As discussed in subsection \ref{boundreprootssec}, if $\alpha$ is a repeated root of $\vn_n(t)$, then $(t-\alpha)^k\exactmid u(t)-\cos(\pi r/n)$, for some $r<n$, co-prime to $n$ and $k\geq 2$. We restrict $D, u,v\in\Qee[t]$, and define $w(t)\in\Qee[t]$ to be the remainder when dividing $u(t)$ by $\pa^k(t)$. That is, $w(t)=u(t)-\pa^k(t)q(t)\in\Qee[t]$, with $\deg{w(t)}<k\da$. If $\deg{w(t)}\neq 0$, we can reduce the problem to looking at repeated factors $(t-\alpha)^k\exactmid w(t)-\cos(\pi r/n)$, instead. Differentiating this divisibility condition gives $(t-\alpha)^{k-1}\exactmid w^{\prime}(t)$, and since $w(t)\in\Qee[t]$ we deduce that $\pa^{k-1}\exactmid w^{\prime}(t)$. This yields a lower bound on the degree of $w(t)$, $\deg{w(t)}\geq (k-1)\da+1$. 
 We first examine the case when $w(t)$ is a constant.

\begin{lemma}
	Let $D(t)\in\Qee[t]$ be a Pellian polynomial, with fundamental solution $(u,v)$. Suppose $\alpha\in\Cee$, algebraic of degree $\da$ with minimal polynomial $\pa(t)$, is a repeated root of $\vn_n(t)$ of multiplicity $k>1$. Then the remainder, when dividing $u(t)$ by $\pa^k(t)$, is a constant if and only if $n=1,\ 2,\ \text{or}\ 3$. 
\end{lemma}

\begin{proof}
Since $\alpha$ is a repeated root of $\vn_n(t)$, of multiplicity $k>1$, we have $(t-\alpha)^k\exactmid u(t)-\cos(\pi r/n)$. From the Euclidean algorithm, there exists a rational polynomial $w(t)$, given by $w(t)=u(t)-\pa^k(t)q(t)$ of degree less than $k\da$.
	Suppose $\deg{w(t)}=0$, then we must have $w(t)=w\in\Qee$. Furthermore, $\cos(\pi r/n)=u(\alpha)=w$ must be a rational number, which is only true for $n=1,\ 2,\ \text{or}\ 3$. Conversely, suppose that $n=1$, then $\vn_1(t)=v(t)$. For a repeated root $\alpha$ of $v(t)$, we must have $\pa^k\exactmid v(t)$. Hence $\pa^{2k}\exactmid u^2(t)-1$, and therefore $\pa^{2k}\exactmid u(t)\pm 1$, and $w(t)=\pm 1$ for all $t$. For $n=2,\ 3$, $\vn_2(t)=u(t)$ and $\vn_3(t)=1\pm 2u(t)$, both polynomials with coefficients in $\Qee$, and therefore $\pa^k\exactmid u(t)$ or $\pa^k\exactmid 1\pm 2u(t)$. Hence $w(t)=0$ or $w(t)=\pm 1/2$, respectively, for all $t$.  
\end{proof}
\begin{cor}
	The polynomials $\vn_2(t)$, $\vn_3(t)$ have repeated complex root $\alpha$ if and only if $u(t)=q(t)\pa^k(t)$ or $u(t)=q(t)\pa^k(t)\pm 1/2$, respectively.
\end{cor}

We can use this corollary to deduce that for any integer $d\geq 1$ and $k>1$, there exists Pellian rational polynomial $D(t)$, such that $\vn_2(t)$ has a repeated factor $p(t)$ of degree $d$. Pick $\alpha\in\Cee$, algebraic with minimal polynomial $p(t)$ of degree $d\geq 1$ over the rationals, and an integer $k>1$. Let $u(t)=p^k(t)$, then $D(t)=u^2-1\in\Qee[t]$ is Pellian, with fundamental solution $(p^k(t), 1)$. Furthermore, $\vn_2(t)=2p^k(t)$, has a repeated factor of degree $d$ and multiplicity $k$. We can do the same for $\vn_3(t)$.  

We have dealt with the case $w(t)=const$, and simultaneously understood the repeated factors of $\vn_n(t)$ for small values of $n$. Therefore, for our investigation into the degree of repeated roots $\alpha\in\Cee$ of $\vn_n(t)$, we assume $n>3$, equivalently  $\deg{w(t)}>0$, and proceed by case analysis.
\subsubsection{The case of an odd degree $\alpha$}\label{dodd}

Suppose $\alpha\in\Cee$ is an algebraic number of degree $\da>1$, an odd integer. Furthermore, let $\alpha$ be a repeated root of $\vn_n(t)$ for $n>3$, then $(t-\alpha)^k\exactmid u(t)-\cos(\pi r/n)$, for some $r<n$, co-prime to $n$ and $k\geq 2$. Then we have the following tower of extensions 

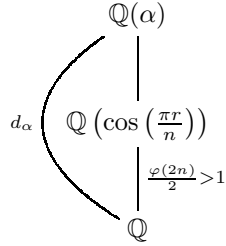
\[
\xymatrix{
	\Qee(\alpha)\ar@{-}[d]^{} \ar@{-}@/_3pc/[dd]_{\da}\\
	\Qee\left(\cos\left(\frac{\pi r}{n}\right)\right)\ar@{-}[d]^{\frac{\varphi(2n)}{2}>1}\\
	\Qee
}
\] 

From the Tower Law, $\da$ must be exactly divisible by $\varphi(2n)/2$. Furthermore, 
\begin{align*}
	\frac{\varphi(2n)}{2}=\begin{cases}
	\varphi(n),\ &\text{if $n$ is even}\\
	\varphi(n)/2,\ &\text{if $n$ is odd.}
	\end{cases}
\end{align*}
Since $\da$ is odd and $\varphi(n)$ is even for all integers $n>3$, the only possibility is for $\varphi(n)/2$ with $n$ odd, to divide $\da$. 
\begin{prop}\label{odd}
Let $D(t)\in\Qee[t]$ be Pellian with fundamental solution $(u,v)$. Suppose that for $n>3$, the polynomial $\vn_n(t)$ has a repeated root $\alpha$ of odd degree $\da$, then $n=q^s$, where $q\equiv 3\mod 4$ is prime and $s$ is a positive integer. Moreover, $\da$ must be a multiple of $\varphi(n)/2$. 
\end{prop}

\begin{proof}

	From the introduction of this section \ref{dodd}, we know that since $\da$ is odd, we must have $n$ and $\varphi(n)/2$ both odd. Lemma \ref{l1}, to follow, implies that $n=q^s$, for a prime $q\equiv 3\mod 4$, and in that case $\varphi(n)/2=(q-1)q^{s-1}/2$, which must divide $\da$. 
\end{proof}


Furthermore, if we only consider prime odd degree, we can say more.
\begin{thm}\label{prime}
Suppose $D(t)\in\Qee[t]$ is Pellian with fundamental solution $(u,v)$. If for any $n>3$, the polynomial $\vn_n(t)$ has a repeated root $\alpha$ of a prime odd degree $\da$, then $n=2\da+1$ is also prime or $n=9$, in which case $\alpha$ is cubic.
\end{thm}

\begin{proof}

	Suppose $\da$ is an odd prime and $n>3$, then we must have $[\Qee(\cos\frac{r\pi}{n}):\Qee]>1$ and thus $\da=\varphi(n)/2$, with $n$ odd. We employ a property of the Euler totient function, which we prove in lemma \ref{l2}, to show that a prime $\da=\varphi(n)/2$ if and only if $n=2\da +1$ or $n=9$ and $\da =3$.
\end{proof}

We now state and prove the technical lemmas on the properties of the Euler totient function needed in the proofs of proposition \ref{odd} and theorem \ref{prime}.
\begin{lemma}\label{l1}
	Suppose $m>3$ is an odd integer. Then $\varphi(m)/2$ is odd if and only if $m=q^s$ with $q\equiv 3 \mod 4$ prime and $s\geq 1$ an integer.
\end{lemma}

\begin{proof}
		Since $m$ is an odd integer, then it can be represented as $\prod_{i=1}^k q_i^{s_i}$, where $q_i$ are distinct odd primes and $s_i$ are positive integers. For each $i$, we have $q_i-1\mid \varphi(m)$. Now if $q_i\equiv 1 \mod 4$, for some $i$, then $\varphi(m)/2$ is even. Hence $q_i\equiv 3\mod 4$ for all $i$. Furthermore, if we have two distinct primes $q_i,q_j$ both dividing $m$, then $(q_i-1)(q_j-1)\mid\varphi(m)$, and once again $\varphi(m)/2$ is even. Therefore, the only possibility for $m$ is to be a power of a prime $q\equiv 3 \mod 4$. 
\end{proof}

\begin{lemma}\label{l2}
	Suppose $m$ is an odd integer greater than $3$, and $p$ is an odd prime. Then $\varphi(m)/2=p$ if and only if $m=2p+1$is prime or $m=9$ and $p=3$.
\end{lemma}

\begin{proof}
	From lemma \ref{l1}, since $p$ is odd we must have $m=q^s$, where $q\equiv 3\mod 4$ prime, and $s$ a positive integer. Therefore $\varphi(m)/2=(q-1)q^{s-1}/2$. For this expression to be equal to the prime $p$, we have two possibilities. Either $q=2p+1$ and $s=1$, yielding $m=2p+1$, or $q=p=3$, $s=2$, giving $m=9$ and $2p+1=7$. In both cases, $2p+1$ is prime. 

	Conversely, if $2p+1$ is prime, then $\varphi(2p+1)=2p$ and $\varphi(9)=2\times3$. Hence $\varphi(m)/2$ is prime for $m=2p+1$ and $m=9$.
\end{proof}

\begin{rk}
	We can use theorem \ref{prime} to discount $\alpha$'s of odd prime degree $\da$. Namely, if $2\da+1$ is not a prime, then $\vn_n(t)$ for $n>3$, has no repeated root of degree $\da$. For example, we cannot have $\alpha$ of degree $7,\ 13,\ 17,\ 19$, etc. Furthermore prime numbers $p$, such that $2p+1$ is prime are called Sophie Germain primes and are a rare occurrence. In particular they form a density $0$ subset of the primes. Therefore, for most primes $p$, there never are repeated roots of $\vn_n(t)$, of degree $p$.
\end{rk}


\subsubsection{The case of $\alpha$, quadratic}

\begin{prop}\label{d}
	Suppose $D(t)\in\Qee[t]$ is Pellian with fundamental solution $(u,v)$. Let $\alpha\in\Cee$ lie in a quadratic extension over the rational numbers. If $\alpha$ is a repeated root of $\vn_n(t)$ for $n>3$, then $n=4,\ 5$ or $6$ and $\alpha\in\Qee(\sqrt{l})$, for $l=2,\ 5$ or $3$, respectively.
\end{prop}

\begin{proof}
	Since $n$ is greater than $3$, $1<[\Qee\left(\cos\left(\frac{r\pi}{n}\right)\right):\Qee]$, and, by the Tower Law, is a factor of $[\Qee(\alpha):\Qee]=2$. In particular,  $\Qee(\alpha)=\Qee\left(\cos\left(\frac{\pi r}{n}\right)\right)$ and we have the following tower of extensions:
		\[
\xymatrix{
	\Qee(\sqrt{l})=\Qee\left(\cos\left(\frac{\pi r}{n}\right)\right)\ar@{-}[d]^{\frac{\varphi(2n)}{2}>1}\ar@{-}@/_3pc/[d]_{2}\\
	\Qee
}
\] 

Therefore, $\varphi(n)=2$ with $n$ even, i.e. $n=4,\ 6$ or $\varphi(n)=4$ with $n$ odd, i.e. $n=5$. Consequently, $\Qee(\alpha)=\Qee(\sqrt{2})$ or $\Qee(\sqrt{3})$, in the former case since the cosines are either $\pm \sqrt{2}/2$ or $\pm\sqrt{3}/2$, respectively. In the latter case, ${\Qee(\alpha)=\Qee(\sqrt{5})}$, since $\cos\frac{\pi r}{5}=\pm(1\pm\sqrt{5})/4$.
\end{proof}

\begin{prop}
Suppose $D(t)\in\Qee[t]$ is Pellian with fundamental solution $(u,v)$. Let $\alpha\in\Cee$ be an algebraic integer with minimal polynomial $\pa$, lying in a quadratic extension over the rational numbers. If $\alpha$ is a repeated root of $\vn_n(t)$ for $n>3$, of multiplicity $k>1$, then $u(t)=g(t)\pa^k(t)+ w(t), \ \text{where}$
	\begin{align*}
		w(t)=\int_{\alpha}^{t}a_k\pa^{k-1}(x) dx +\cos\frac{\pi r}{n},
	\end{align*}
and
\begin{align*}
		a_k=\frac{(2k-1)!g}{(-4l)^{k-1}s^{2k-1}((k-1)!)^2},
	\end{align*}
with $g=\pm1/2$, when $l=2,\ 3$ and $g=\pm 1/4$, when $l=5$. And $s$ is a quantity which can be determined from the computation in the proof. 

\end{prop}

\begin{proof}
Let $\alpha$ be as in the statement of the theorem, with minimal polynomial $\pa(t)$ of degree $\da=2$. From the discussion at the beginning of subsection \ref{secrepfacdeg}, we know $(k-1)\da+1\leq \deg{w(t)}<k\da$.  Therefore, for $\alpha$ a quadratic irrational, $\deg w(t)=2k-1$. In addition, since $w^{\prime}(t)\in\Qee[t]$ and $(t-\alpha)^{k}$ is a factor of $w(t)-\cos(\pi r /n)$, we must have $\pa^{k-1}\exactmid w^{\prime}(t)$, and thus ${\deg \pa^{k-1}(t)=2k-2=\deg w^{\prime}(t)}$. Hence
	\begin{align*}
 	w^{\prime}(t)&=a_k\pa^{k-1}(t), \ \text{for}\ a_k\in\Qee^*,\ \text{and} \\
    w(\alpha)    &=\cos\frac{\pi r}{n}.
 \end{align*}
Equivalently,
\begin{align*}
	w(t)=\int_{\alpha}^{t} a_k\pa^{k-1}(t) dx+\cos\frac{\pi r}{n}.
\end{align*}
To completely determine $w(t)$, it remains to compute the coefficient $a_k$.

Suppose that $\alpha$ has minimal polynomial $\pa(t)=t^2+2bt+c\in\Qee[t]$. Then using $T=t+b$, we can rewrite it as $P_{\alpha}(T)=T^2-A$, where $A=b^2-c$. From Proposition \ref{d} we know that $\Qee(\alpha)=\Qee(\sqrt{l})$, for $l=2,\ 3$ or $5$. Hence $A=s^2l$, for some rational number $s$, and
 \begin{align*}
 	w(t)
 	    &=\int_{\alpha+b}^{t+b}a_k(X^2-A)^{k-1} dX + \cos\frac{\pi r}{n}\\
 	     &=a_k\int_{s\sqrt{l}}^{t+b} \sum_{j=0}^{k-1}\binom{k-1}{j}X^{2j}(-s^2l)^{k-1-j}dX + \cos\frac{\pi r}{n}\\
 	     &=a_k\sum_{j=0}^{k-1}\binom{k-1}{j}(-s^2l)^{k-1-j}\left[\frac{(t+b)^{2j+1}}{2j+1}-\frac{s^{2j+1}l^j\sqrt{l}}{2j+1}\right]+ \cos\frac{\pi r}{n}.\\
 \end{align*}
 Furthermore, $\cos\frac{\pi r}{n}$ is also in a quadratic extension of the rationals, so let it be of the form $h+g\sqrt{l}$, with $h,g\in\Qee$. Now, $w(t)$ is a rational polynomial and thus the value of $g$ will be such that it cancels the coefficient of $\sqrt{d}$ in the sum above. Namely,
\begin{align*}
	 g=& a_ks^{2k-1}l^{k-1}\sum_{j=0}^{k-1}\binom{k-1}{j}\frac{(-1)^{k-1-j}}{2j+1}.
\end{align*}
In a lemma given after the proof, we show that the sum in the expression for $g$ evaluates to $$\frac{(-4)^{k-1}(k-1)!}{(2k-1)!}.$$ After rearranging appropriately, we obtain the required form for $a_k$.
\end{proof}

\begin{lemma}
	For a positive integer $n$, we have the following combinatorial identity $$\sum_{j=0}^{n}\binom{n}{j}\frac{(-1)^{n-j}}{2j+1}=\frac{(-4)^{n}(n!)^2}{(2n+1)!}.$$
\end{lemma}

\begin{proof}

Let $f(n)=\int_0^1 (x^2-1)^n dx$. After expanding binomially, we get 
\begin{align*}
	f(n)=\int_0^1 (x^2-1)^n dx&=\int_0^1\sum_{j=0}^{n}\binom{n}{j}(-1)^{n-j}x^{2j} dx\\
&=\sum_{j=0}^{n}\binom{n}{j}\frac{(-1)^{n-j}}{2j+1}.
\end{align*}

Integrating $f(n)$ by parts, we obtain a recursive relation
\begin{align*}
	f(n)
        &=[x(x^2-1)^n]_0^1-\int_0^1 2nx^2(x^2-1)^{n-1} dx\\
        &=-2n\int_0^1 x^2(x^2-1)^{n-1} dx\\
        &=-2n\left(f(n)+f(n-1)\right).
\end{align*}

Therefore
\begin{align*}
	   f(n)=\frac{-2n}{2n+1}f(n-1)&=f(0)\prod_{i=1}^n \frac{(-1)^i2i}{2i+1}
	       =\frac{(-4)^{n}(n!)^2}{(2n+1)!}.
\end{align*}
The final equality follows since $f(0)=1$.
\end{proof}

\begin{cor}\label{integral}
	The polynomials $\vn_n(t)$ for $n>3$ and $u(t)\in\Zee[t]$, have no quadratic irrationals as repeated roots of multiplicity $k>1$.
\end{cor}

\begin{proof}
	If $\alpha$ is as in the statement and $u(t)\in\Zee[t]$, then $w(t)\in\Zee[t]$. This implies that $a_k\in\Zee$. However, for $k>1$, the power of $2$ dividing $$\frac{(2k-1)!}{4^{k-1}((k-1)!)^2}$$ is smaller than $0$. To see this, we rearrange the expression to get
	\begin{align*}
	\frac{2k-1}{2^{2k-2}}\binom{2k-2}{k-1}.
	\end{align*}

We apply Kummer's theorem which says that for a prime $p$, $p^l\mid\binom{n}{m}$ only if $p^l\leq n$. And $2^{2k-2}>2k-2$ for $k>1$, hence $a_k$ is not an integer for any integer $k$ greater than 1. 
\end{proof}

\section*{Appendix}
 \subsection*{Non square-free Pellian polynomials}
It is well-known that for every positive integer $D$, $x^2-Dy^2=1$ has non-trivial solutions, but this no longer holds true for every polynomial $D(t)\in\Cee[t]$. In general, it is fairly difficult to determine whether a given complex polynomial is Pellian or not. However, for a polynomial with roots of high multiplicity, as a consequence of the ABC theorem \cite{Mason1984}\cite{Stothers1981}, Dubickas and Steuding in \cite{Dubickas2004} give an easy criterion that we can check.



\begin{thm}\label{nonpell}
	If the number $n(D)$ of distinct zeros of $D\in\Cee[x]$ is less than or equal to $\frac{1}{2}\deg D$, then the polynomial Pell equation has no non-trivial solutions in $\Cee[x]$.
\end{thm}


Observe that given a separable polynomial $F(t)\in\Cee[t]$, and a square-free polynomial $D(t)\in\Cee[t]$, both of positive degree, and relatively prime, then $n(F^2D(t))=\deg{F(t)}+\deg{D(t)}>\frac{1}{2}\deg{F^2D(t)}$. Hence, for polynomials with coefficients in $\Cee$, and a single square factor, Theorem \ref{nonpell} cannot be used to determine whether they are Pellian or not. We thus focus our attention on polynomials of that form.

 Suppose that $F^2D(t)$ is a Pellian with complex coefficients. Then there exist polynomials $X(t),Y(t)\in\Cee[t]$ solving the corresponding Pell's equation. Moreover, $(X(t),FY(t))$ solves Pell's equation for $D(t)$. On the other hand we have the following.
 \begin{lemma}\label{converse}
 	If $D(t)\in\Cee[t]$ is Pellian with solutions $(u_n(t),v_n(t))$ then $\Delta(t)=F^2D(t)$ is also Pellian, if and only if $F(t)\mid v_n(t)$ for some $n$. 
 \end{lemma}

\begin{proof}
	If $V_n(t)=Fv_n(t)$ for some $n$, then $(u_n(t), V_n(t))$ is a solution to Pell for $\Delta(t)$, and thus $\Delta(t)$ is Pellian. The converse follows if we argue by contradiction and use that if $(u(t),v(t))$ is a solution to Pell for $D(t)$, then $(u(t),Fv(t))$ is a solution to Pell for $F^2D(t)$.
\end{proof}
Therefore, all Pellian polynomials of the form $F^2D$, arise from a square-free Pellian polynomial $D(t)$ and a factor $F$ of $\vn_n(t)$. This lemma gives a simple method for checking whether a polynomial $F^2D\in\Qee[t]$ is Pellian or not. Furthermore,




\begin{prop}\label{construct}
Let $D(t)\in\Qee[t]$ be  square-free and Pellian. Then for a given positive integer $f$, there exist only finitely many irreducible $F\in\Qee[t]$, of degree $f$, such that $F^2D$ is also Pellian.
\end{prop}

\begin{proof}
	From lemma \ref{converse}, $F$ must be a factor of $v_n(t)$. Furthermore, any such factor arises from a factor of $\vn_n(t)$. By corollary \ref{factors}, there are only finitely many such factors of a fixed degree. 
\end{proof}
\begin{eg}
	Let $D(t)=t^2-1$ and we wish to find all quadratic polynomials $F\in\Qee[t]$ such that $F^2(t^2-1)$ is Pellian. These polynomials must be factors of $\vn_n(t)$ for some $n$. From theorem \ref{linear}, we should only look at $n\leq 6$. Therefore, the only quadratic polynomials $F$, for which $F^2(t^2-1)$ is Pellian, are $\vn_4(t),\ \vn_6(t)$ and the factors of $\vn_5(t)$. Respectively, these are given by
	\begin{align*}
	 	2t^2-1,\ 4t^2-3,\text{ and } 4t^2\pm 2t-1,
	 \end{align*} 
	 respectively.
\end{eg}

This result comes as a contrast to the classical case, where for any positive integer $d$, Pell's equation for $g^2d$ has non-trivial solutions for infinitely many $g$. For details see chapter 8 of \cite{LeVeque2012}.

\bibliographystyle{plain}
\bibliography{Arxiv-reproots}

\end{document}